\documentclass[12pt]{amsart}
\usepackage{amsmath,amssymb,amsthm,hyperref,url,mathtools}
\usepackage{amsfonts,amsmath,mathrsfs,amssymb,graphicx,multirow,amsthm,latexsym,verbatim,tikz}

\usepackage[mathscr]{euscript}
\usepackage{bm}

\newcommand{\ZZ}{\mathbb Z}

\newcommand{\CC}{\mathbb C}

\theoremstyle{plain}
\newtheorem{theorem}{Theorem}

\newtheorem{prop}[theorem]{Proposition}
\newtheorem{cor}[theorem]{Corollary}
\newtheorem{lemma}[theorem]{Lemma}

\theoremstyle{definition}
\newtheorem{definition}[theorem]{Definition}
\newtheorem{notation}[theorem]{Notation}

\newtheorem{example}[theorem]{Example}
\newtheorem{rem}[theorem]{Remark}

\numberwithin{theorem}{section}
\numberwithin{equation}{section}

\title{Syzygies of Some Invariant Rings with Cyclic Group Actions}
\author[]{Christin Sum}

\newcommand{\inva}{\text{inv}}
\newcommand{\invab}{\text{inv}_{1,b}}

\newcommand{\inv}{^{-1}}
\newcommand{\sg}{S^G_{1,b}}
\newcommand{\x}{x_1}
\newcommand{\xx}{x_2}
\newcommand{\varphiab}{\varphi_{a,b}}
\newcommand{\varphib}{\varphi_{1,b}}

\usepackage{setspace}

\begin{document}
\begin{abstract}
    We investigate actions of cyclic groups on polynomial rings
    with two variables, and the minimal free resolution of the
    corresponding invariant ring. 
    In particular, we fully classify several cases, including the case the defining ideal has codimension two, and when the multidegrees of generating invariants lie on the union of two lines.
\end{abstract}

\maketitle

\section{Introduction}

The late nineteenth century witnessed breakthroughs from Hilbert \cite{ueberformen} \cite{ueberinvariant} in invariant theory. 
Notably, Hilbert's Basis Theorem solved the fundamental problem of invariant theory concerning the finitely generated nature of invariant rings. 
Within the same era, Hilbert's Syzygy Theorem proved that the free resolution of a finitely generated module over a polynomial ring in $n$ variables has length at most $n$.

In recent decades, research in invariant theory has broadened to other branches of mathematics, some of which include geometric invariant theory \cite{mumford}, computational invariant theory \cite{sturmfels}, and connections between invariant theory and combinatorics \cite{stanley}. 
Stanley in particular examined invariants of finite groups and their applications to combinatorics.

This paper investigates the invariant ring of cyclic groups acting on a polynomial ring with two variables and aims to explicitly construct its minimal free resolution. 
Our main result provides the free resolution when the invariant ring has four generators, and determines the conditions that yield such an invariant ring. 
We then expand to five generators, and end in the case where the invariant ring has a simple structure in a different sense. 

Additionally, this work falls in line with the recent surge of activity in the study of syzygies in the nonstandard graded setting. Examples of such studies include 
Benson's \cite{benson} paper that generalizes Castelnuovo-Mumford regularity to nonstandard $\ZZ$-graded polynomial rings, Symonds's \cite{symonds10},\cite{symonds11} resulting work in invariant theory, and more
\cite{brown2022tateresolutionstoricvarieties},\cite{brown2023linearsyzygiescurvesweighted},\cite{brownerm24},\cite{davis2024rationalnormalcurvesweighted}.

\medskip

Throughout this paper, let $S = \CC[x_1,x_2]$ and $G= \ZZ/p\ZZ = \langle \zeta = e^{2\pi i/p} \rangle$ where $p$ is prime. 
Let $0 < a,b < p$ be integers that define an action of $G$ on $S$ via
\begin{align*}
    x_1 &\mapsto \zeta^a x_1
    \\
    x_2 &\mapsto \zeta^bx_2
\end{align*}
This yields a finite set $\{z_0, z_1, \dots, z_n\}$ of generators of the invariant ring $S^G_{a,b}$ as a $\CC$-algebra, which we will write as $inv_{a,b}$ and arrange in lexicographic order.
We then set $R = \CC[y_0, y_1, \dots, y_n]$ with $\deg y_i = \deg z_i$, along with the map
\begin{align*}
    \varphiab : R &\rightarrow S^G_{a,b}
    \\
    y_i &\mapsto z_i
\end{align*}
so that $S^G_{a,b} \cong R / \ker \varphiab$.
Within these settings, our first main result explores the case when 
$S^G_{a,b}$ is generated by precisely four invariants. 
Before stating this result, let us introduce some relevant notation.

\begin{notation}
    For an integer $c$, we will write $c_p$ for the unique integer $0 \leq c_p < p$ such that $c$ is congruent to $c_p \text{ mod } p$.
    Additionally, we will write $c \inv$ for the unique integer $0 < c^{-1} < p$ such that $cc^{-1}$ is congruent to $1 \text{ mod } p$.
    We will also often write $\equiv_p$ for congruence modulo $p$.
\end{notation}

\begin{theorem}\label{thm1}
    With notation as above, the following are equivalent:
    \begin{enumerate}
        \item $S^G_{a,b}$ is generated by 4 invariants (i.e. codim $\ker \varphiab =2$).
        \item $(p-(a\inv b)_p)(p-(ab\inv)_p) = p+1.$
    \end{enumerate}
\end{theorem}

In the context of Theorem \ref{thm1}, $S^{G}_{a,b}$ is a Cohen-Macaulay, codimension 2 algebra and hence its minimal free resolution can be understood via the Hilbert-Burch Theorem(\cite{eisenbudcommalg}, Theorem 20.15). 
Corollary \ref{4 inv res} provides an explicit description of this minimal free resolution in terms of the integers $p,a,b$.

Our next main result expands to the five invariant case, proving:

\begin{prop}\label{prop2}
    With notation as above,  if $(p-(a\inv b)_p)(p-(ab\inv)_p) = 2p+1$ then $S^G_{a,b}$ is generated by 5 invariants (i.e. codim $\ker \varphiab=3$). 
\end{prop}

Unlike Theorem \ref{thm1}, the converse of this proposition is not guaranteed (see Example \ref{ex2}). 
Moreover, in Propositions \ref{2p+1 lower ker} and \ref{2p+1 upper ker} we show that $\ker \varphiab$ is generated by the maximal minors of a matrix in the five invariant case, from which we conclude the free resolution of $S^G_{a,b}$ is given by an Eagon-Northcott complex.

We then veer from cases that consider the number of generating invariants to a particular case that determines when the multidegrees of generating invariants live on precisely two lines. 
We classify conditions on $p,a,b$ under which this case occurs. 
These resolutions are not necessarily presented by ideals of maximal minors of the expected codimension.

This paper is broken up as follows: In $\S$\ref{section2}, we provide examples relevant to Theorem \ref{thm1} and Proposition \ref{prop2}.
In $\S$\ref{section3}, we prove an equivalent statement of Theorem \ref{thm1} and in $\S$\ref{section4} we provide the corresponding free resolution. 
In $\S$\ref{section5}, we prove an equivalent statement of Proposition \ref{prop2}. 
We conclude in $\S$\ref{section6}, classifying when the multidegrees of invariants live on two lines.

\smallskip

\noindent \textbf{Acknowledgments.} I would like to thank my advisor, Daniel Erman, for all of his guidance and support. The majority of these results came from observations made while using Macaulay2 \cite{M2}.

\medskip

\section{Examples}\label{section2}

Before proving our main results, let us illustrate with some examples. 
First consider the following example that satisfies the conditions of Theorem \ref{thm1}.

\begin{example}\label{ex1}
    Let $G = \ZZ/7\ZZ$ and $a=1, b=3$.
    That is, $G = \langle \zeta = e^{2\pi i /7} \rangle$ acts on $S = \CC[x_1,x_2]$ via
    \begin{align*}
        x_1 &\mapsto \zeta x_1
        \\
        x_2 &\mapsto \zeta^3 x_2.
    \end{align*}
    It follows that $S^G_{1,3}$ is generated by $\inva_{1,3} = \{x_1^7, \; x_1^4x_2, \; x_1x_2^2, \; x_2^7\}$.
    Thus we set
    $R = \CC[y_0,y_1,y_2,y_3]$ with $\deg y_0 = 7, \deg y_1 = 5, \deg y_2 = 3, \deg y_3 = 7$, and the map
    \begin{align*}
    \varphi_{1,3} : R &\rightarrow S^G
    \\
    y_0 &\mapsto x_1^7
    \\
    y_1 &\mapsto x_1^4x_2
    \\
    y_2 &\mapsto x_1x_2^2
    \\
    y_3 &\mapsto x_2^7.
    \end{align*}
    In this case, $\ker \varphi_{1,3} = \langle y_1^2 - y_0y_2, \; y_2^4 - y_1y_3, \; y_1y_2^3 - y_0y_3 \rangle$ and the free resolution of $R^1/\ker \varphi_{1,3}$ is
    $$R^1
    \xlongleftarrow{[\ker \varphi_{1,3}]}
    \renewcommand*{\arraystretch}{1.1}
    \begin{tabular}{ c }
        $R^1(-10)$
        \\
        $\bigoplus$
        \\
        $R^1(-12)$
        \\
        $\bigoplus$
        \\
        $R^1(-14)$
    \end{tabular}
    \xlongleftarrow{
    \begin{bmatrix}
        -y_3 & -y_2^3
        \\
        -y_1 & -y_0
        \\
        y_2 & y_1
    \end{bmatrix}} 
    \renewcommand*{\arraystretch}{1.1}
    \begin{tabular}{ c }
        $R^1(-17)$
        \\
        $\bigoplus$
        \\
        $R^1(-19)$
    \end{tabular}
    $$
    Furthermore, since $p=7, \; a=1, \; b=3$, then $a\inv = 1, \; b\inv = 5$.
    Hence
    $$(p-(a^{-1}b)_p)(p-(ab^{-1})_p) = (7-3)(7-5) = 8 = p+1.$$
    That is, $(p-(a^{-1}b)_p)(p-(ab^{-1})_p) = p+1$ where
    $S^G_{1,3}$ is also generated by 4 invariants.
    So this example satisfies the conditions of Theorem \ref{thm1}.
\end{example}

Now observe the following example which does not satisfy any of the conditions of Theorem \ref{thm1} and also demonstrates the nonequivalence of Proposition \ref{prop2}.
\begin{example}\label{ex2}
    Let $G = \ZZ/13\ZZ$ and $a=7, b=9$.
    Thus $\inva_{7,9} = \{x_1^{13}, x_1^{8}x_2, x_1^3x_2^2, x_1x_2^{5}, x_2^{13}\}$ and $R = \CC[y_0, y_1,y_2,y_3,y_4]$ with $\deg y_0 = 13, \deg y_1 = 9, \deg y_2 = 5, \deg y_3 = 6, \deg y_4 = 13$. With this,
    $$\ker \varphi_{7,9} = \langle y_2^3-y_1y_3, y_3^3-y_2y_4, y_1^2-y_0y_2,y_1y_2^2-y_0y_3, y_2^2y_3^2-y_1y_4,y_1y_2y_3^2-y_0y_4\rangle,$$ 
    and the free resolution of $R^1/\ker \varphi_{7,9}$ is
    $$R^1
    \xlongleftarrow{[\ker \varphi_{7,9}]}
    \renewcommand*{\arraystretch}{1.1}
    \begin{tabular}{ c }
        $R^1(-15)$
        \\
        $\bigoplus$
        \\
        $R^2(-18)$
        \\
        $\bigoplus$
        \\
        $R^1(-19)$
        \\
        $\bigoplus$
        \\
        $R^1(-22)$
        \\
        $\bigoplus$
        \\
        $R^1(-26)$
    \end{tabular}
    \xlongleftarrow{
    d_1
    }
    \renewcommand*{\arraystretch}{1.1}
    \begin{tabular}{ c }
        $R^1(-24)$
        \\
        $\bigoplus$
        \\
        $R^1(-27)$
        \\
        $\bigoplus$
        \\
        $R^2(-28)$
        \\
        $\bigoplus$
        \\
        $R^2(-31)$
        \\
        $\bigoplus$
        \\
        $R^1(-32)$
        \\
        $\bigoplus$
        \\
        $R^1(-35)$
    \end{tabular}
    \xlongleftarrow{
    d_2
    }
    \renewcommand*{\arraystretch}{1.1}
    \begin{tabular}{ c }
        $R^1(-37)$
        \\
        $\bigoplus$
        \\
        $R^1(-40)$
        \\
        $\bigoplus$
        \\
        $R^1(-41)$
    \end{tabular}
    $$
    where 
    
    $d_1 = 
    \renewcommand*{\arraystretch}{0.9}
    \begin{bmatrix}
        -y_1 & -y_3^2 & -y_0 & -y_4 & 0 & 0 & 0 & 0 
        \\
        -y_3 & 0 & -y_2^2 & 0 & y_4 & -y_4 & 0 & -y_2y_3^2
        \\
        0 &-y_1&0&-y_2^2&-y_0 &  0 & -y_1y_2 & 0 
        \\
        y_2 & 0 & y_1 & 0 & -y_3^2 & 0 & -y_4 & 0 
        \\
        0   & y_2  &  0  & y_3 & y_1 & -y_1 & 0 & -y_0
        \\
        0 & 0 & 0 & 0 & 0 & y_2 & y_3 & y_1
    \end{bmatrix}$
    and $d_2 =
    \renewcommand*{\arraystretch}{0.9}
    \begin{bmatrix}
        y_4 & 0 & -y_2y_3^2
        \\
        0  & y_0 & y_1y_2
        \\
        0 & -y_3^2 & -y_4
        \\
        -y_1 & 0 & y_0
        \\
        0 & -y_1 & -y_2^2 
        \\
        -y_3 & -y_1 & 0
        \\
        y_2 & 0 & -y_1
        \\
        0 & y_2 & y_3
    \end{bmatrix}.$
    In this case, as $p=13, a=7, b=9$, then $a\inv = 2, b\inv = 3$.
    Hence
    $$(p-(a^{-1}b)_p)(p-(ab^{-1})_p) = (13-18_{13})(13-21_{13}) = (13-5)(13-8) = 40 = 3p+1.$$
    That is, $(p-(a^{-1}b)_p)(p-(ab^{-1})_p) = 3p+1$.
    Meanwhile $S^G_{7,9}$ is generated by 5 invariants, thus showing the converse of Proposition \ref{prop2} is not guaranteed.
\end{example}

We conclude with an example that satisfies Proposition \ref{prop2}.

\begin{example}\label{ex3}
    Let $G = \ZZ/13\ZZ$ and $a= 1,b=4$. Thus $\inva_{1,4} = \{x_1^{13}, x_1^9x_2, x_1^5x_2^2, x_1x_2^3, x_2^{13} \}$ and $R = \CC[y_0, y_1,y_2,y_3,y_4]$ with $\deg y_0 = 13, \deg y_1 = 10, \deg y_2 = 7, \deg y_3 = 4, \deg y_4 = 13$. Therefore
    $$\ker \varphi_{1,4} = \langle y_2^2 - y_1y_3, y_1y_2 - y_0y_3, y_3^5 - y_2y_4, y_1^2 - y_0y_2, y_2y_3^4 - y_1y_4, y_1y_3^4 - y_0y_4 \rangle,$$
    and the free resolution of $R^1 / \ker \varphi_{1,4}$ is
    $$R^1
    \xlongleftarrow{[\ker \varphi_{1,4}]}
    \renewcommand*{\arraystretch}{1.1}
    \begin{tabular}{ c }
        $R^1(-14)$
        \\
        $\bigoplus$
        \\
        $R^1(-17)$
        \\
        $\bigoplus$
        \\
        $R^2(-20)$
        \\
        $\bigoplus$
        \\
        $R^1(-23)$
        \\
        $\bigoplus$
        \\
        $R^1(-26)$
    \end{tabular}
    \xlongleftarrow{
    d_1
    }
    \renewcommand*{\arraystretch}{1.1}
    \begin{tabular}{ c }
        $R^1(-24)$
        \\
        $\bigoplus$
        \\
        $R^2(-27)$
        \\
        $\bigoplus$
        \\
        $R^2(-30)$
        \\
        $\bigoplus$
        \\
        $R^2(-33)$
        \\
        $\bigoplus$
        \\
        $R^1(-36)$
    \end{tabular}
    \xlongleftarrow{
    d_2
    }
    \renewcommand*{\arraystretch}{1.1}
    \begin{tabular}{ c }
        $R^1(-37)$
        \\
        $\bigoplus$
        \\
        $R^1(-40)$
        \\
        $\bigoplus$
        \\
        $R^1(-43)$
    \end{tabular}
    $$
    where
    \\
    $d_1 = 
    \renewcommand*{\arraystretch}{0.9}
    \begin{bmatrix}
        -y_1 & y_0 & -y_4 & -y_3^4 & 0 & 0 & 0 & 0 
        \\
        y_2 & -y_1 & 0 & y_4 & -y_4 & -y_3^4 & 0 & 0
        \\
        -y_3 & y_2 & 0 & 0 & 0 & y_4 & -y_4 & -y_3^4
        \\
        0 & 0 & -y_2 & 0 & -y_1 & -y_0 & 0 & 0
        \\
        0   & 0  &  y_3 & y_2 & 0 & y_1 & -y_1 & -y_0
        \\
        0 & 0 & 0 & -y_3 & y_3 & 0 & y_2 & y_1
    \end{bmatrix}$
    and $d_2 =
    \renewcommand*{\arraystretch}{0.9}
    \begin{bmatrix}
        -y_4 & y_3^4 & 0 
        \\
        0  & -y_4 & y_3^4
        \\
        y_1 & -y_0 & 0   
        \\
        0   & -y_1 & y_0 
        \\
        -y_2 & 0   &  y_0 
        \\
        0  &  y_2 &  -y_1  
        \\
        y_3  & 0   &  -y_1 
        \\
        0   & -y_3  & y_2 
    \end{bmatrix}.$
    Furthermore, since $p=13, a=1,b=4$, then $a\inv = 1, b\inv = 10$.
    Thus
    $$(p-(a^{-1}b)_p)(p-(ab^{-1})_p) = (13-4)(13-10) = 27 = 2p+1.$$
    That is, $(p-(a^{-1}b)_p)(p-(ab^{-1})_p) = 2p+1$ where
    $S^G_{1,4}$ is also generated by 5 invariants. 
    So this example satisfies Proposition \ref{prop2}.
    
    Moreover, note that the multidegrees of the generating invariants are $$\{(13,0),(9,1),(5,2),(1,3),(0,13)\}$$
    with consecutive slopes $-\frac{1}{4}$ and $-10$, as depicted in the following graph: 
    \begin{center}
        \includegraphics[scale=0.21]{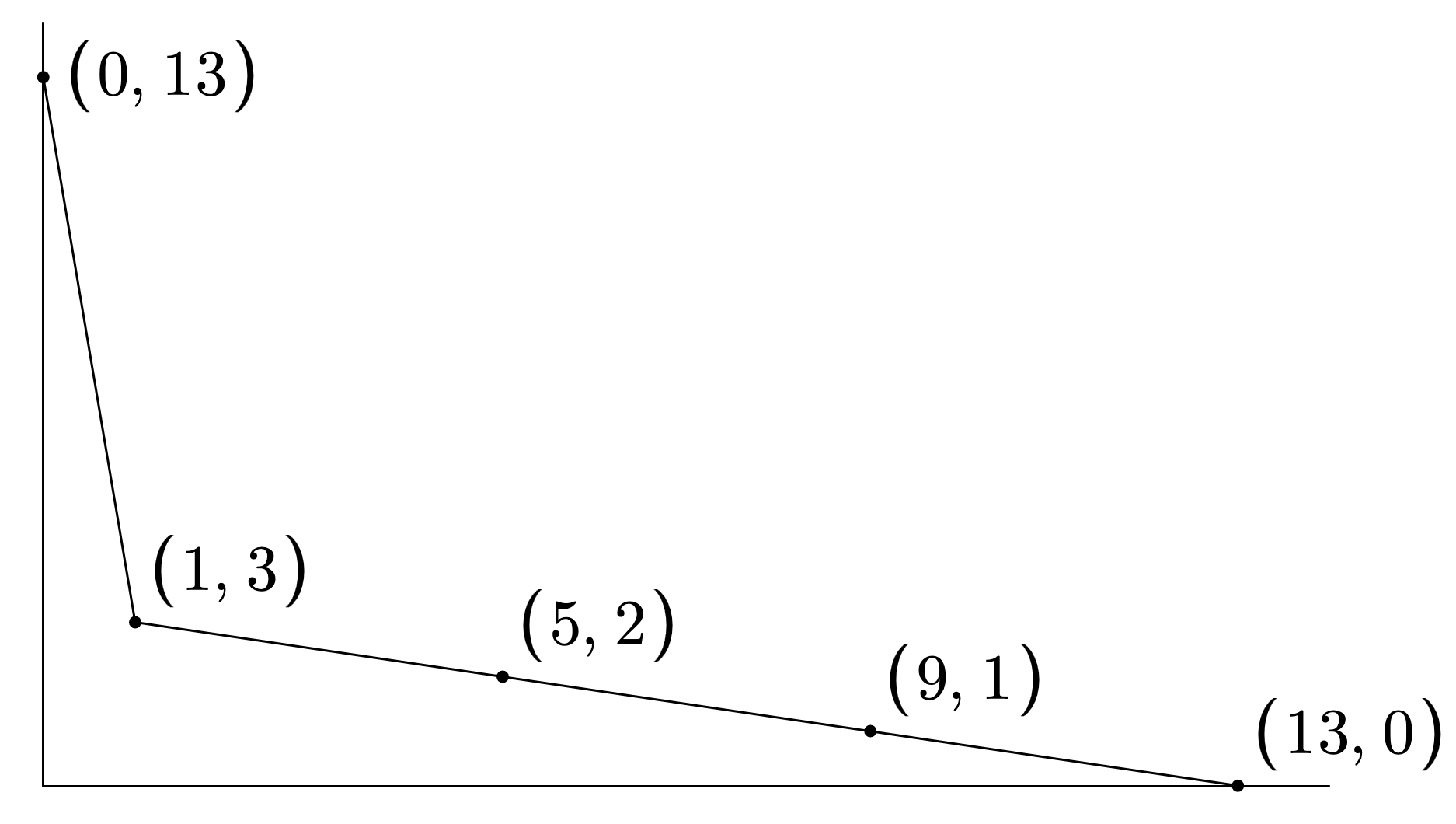}
    \end{center}
    The multidegrees of the invariants thus lie on two lines, and we will elaborate more on this in $\S$\ref{section6}.
\end{example}


\medskip

\section{Proof of Theorem 1.2}\label{section3}
To prove Theorem \ref{thm1}, we will introduce some propositions that allow us to prove an equivalent statement of our theorem.

\begin{prop}
      $\inva_{a,b} = \inva_{1,a\inv b}$.
\end{prop}

\begin{proof} 
    Since $\inva_{a,b}$ and $\inva_{1,a\inv b}$ are the respective generating sets of $S^G_{a,b}$ and $S^G_{1,a\inv b}$, it suffices to prove $S^G_{a,b} = S^G_{1,a\inv b}$. Consider an arbitrary monomial $x_1^cx_2^d$:
    \begin{align*}
        x_1^cx_2^d \in S^G_{a,b}
        &\Leftrightarrow
        (\zeta^a x_1)^c (\zeta^b x_2)^d = x_1^cx_2^d
        \\
        &\Leftrightarrow
        \zeta^{ac+bd}x_1^cx_2^d = x_1^cx_2^d
        \\
        &\Leftrightarrow
        ac+bd \equiv_p 0 
        \\
        &\Leftrightarrow
        c + a\inv bd \equiv_p 0 
        \\
        &\Leftrightarrow
        \zeta^{c+a\inv bd}x_1^cx_2^d = x_1^cx_2^d
        \\
        &\Leftrightarrow
        (\zeta x_1)^c (\zeta^{a\inv b}x_2)^d = x_1^cx_2^d
        \\
        &\Leftrightarrow
        x_1^cx_2^d \in S^G_{1,a\inv b}.
    \end{align*}
    As the monomials in $S^G_{a,b}$ are equal to those in $S^G_{1,a\inv b}$, then the invariant rings are equal and thus their generating sets are likewise equal.
\end{proof}

\begin{rem}
    As a result of this proposition, we only need to consider actions 
    \begin{align*}
        x_1 &\mapsto \zeta^a x_1
        \\
        x_2 &\mapsto \zeta^b x_2
    \end{align*}
   where $a=1$ and $0 < b < p$.
\end{rem}
\begin{prop}
    There is a natural bijection between $\inva_{1,b}$ and $\inva_{1,b\inv}$ induced by transposing $x_1$ and $x_2$, and this induces a ring isomorphism between $S^G_{1,b}$ and $S^G_{1,b^{-1}}$.
\end{prop}

\begin{proof}
    Consider an arbitrary monomial $x_1^cx_2^d$:
    \begin{align*}
        x_1^cx_2^d \in \sg
        &\Leftrightarrow
        (\zeta \x)^c(\zeta^b\xx)^d = \x^c\xx^d
        \\
        &\Leftrightarrow
        \zeta^{c+bd}x_1^cx_2^d = x_1^cx_2^d
        \\
        &\Leftrightarrow
        c+bd \equiv_p 0 
        \\
        &\Leftrightarrow
        b\inv c + d \equiv_p 0 
        \\
        &\Leftrightarrow
        \zeta^{d+b\inv c}x_1^dx_2^c = x_1^dx_2^c
        \\
        &\Leftrightarrow
        (\zeta \x)^d(\zeta^{b\inv}\xx)^c = \x^d\xx^c
        \\
        &\Leftrightarrow
        x_1^dx_2^c \in S^G_{1,b\inv}.
    \end{align*}
    Hence $S^G_{1,b} \cong S^G_{1,b\inv}$ follows.
\end{proof}

\begin{rem}
    Thus we further only need to consider actions 
    \begin{align*}
        x_1 &\mapsto \zeta x_1
        \\
        x_2 &\mapsto \zeta^b x_2
    \end{align*}
    where $0 < b \leq b\inv < p$.
\end{rem}

Hence we have reduced Theorem \ref{thm1} to proving the following theorem:
\begin{theorem}\label{thm2}
    Let $G = \ZZ/p \ZZ = \langle \zeta = e^{2\pi i/p} \rangle$ act on $S = \CC[x_1,x_2]$ via
    \begin{align*}
        x_1 &\mapsto \zeta x_1
        \\
        x_2 &\mapsto \zeta^b x_2
    \end{align*}
    where $0 < b  \leq b\inv < p$.
    The following are equivalent:
    \begin{enumerate}
        \item $S^G_{1,b}$ is generated by 4 invariants (i.e. codim $\ker \varphiab =2$).
        \item $(p-b)(p-b\inv) = p+1.$
    \end{enumerate}
\end{theorem}

We continue with lemmas regarding the generating set $\invab$ of $\sg$.

\begin{lemma}\label{invariant-lemmas}
    Let $x_1^cx_2^d \in \invab$.
    \begin{enumerate}
        \item $c,d \leq p$.
        \item If $x_1^ex_2^f \in S^G_{1,b}$ and $c \geq e$ then $d \leq f$.
        \item If $x_1^ex_2^f \in S^G_{1,b}$ and $d \geq f$ then $c \leq e$.
        \item If $x_1^ex_2^f \in \invab$ and $c=e$ or $d=f$ then $x_1^cx_2^d = x_1^ex_2^f$.
    \end{enumerate}
\end{lemma}

\begin{proof}
    For (1), by way of contradiction suppose $c > p$. 
    Thus $c = p + r$ for some positive integer $r$. 
    Since $x_1^cx_2^d \in \invab$ then $0 \equiv_p c+bd \equiv_p (p+r) + bd \equiv_p r+ bd$. 
    Hence $x_1^rx_2^d \in \sg$. 
    So $\x^c \xx^d = \x^{p+r}\xx^d = (\x^p)(\x^r\xx^d)$ where $\x^p, \x^r\xx^d \in \sg$. 
    Thus $\x^p$ and $\x^r \xx^d$ divide $\x^c \xx^d$ in $\sg$, a contradiction to $\x^c \xx^d \in \invab$. 
    Therefore $c \leq p$ indeed holds.
    A similar argument proves $d \leq p$.

    For (2), suppose $x_1^ex_2^f \in \sg$ and $c \geq e$. By way of contradiction suppose $d > f$. 
    Thus $d = f + s$ for some integer $s > 0$. 
    Similarly since $c \geq e$ then $c = e + r$ for some integer $r \geq 0$. 
    Note that as $x_1^cx_2^d, x_1^ex_2^f \in \sg$, then by definition $c + bd \equiv_p 0 \equiv_p e + bf$. 
    This further yields
    \begin{align*}
        0 \equiv_p c+bd \equiv_p (e+r) + b(f+s) \equiv_p (e+bf) + (r+bs) \equiv_p r+bs.
    \end{align*}
    Hence $x_1^rx_2^s \in \sg$. 
    So $x_1^cx_2^d = x_1^{e+r}x_2^{f+s} = (x_1^ex_2^f)(x_1^rx_2^s)$ where $x_1^ex_2^f, x_1^rx_2^s \in \sg$. 
    Thus $x_1^ex_2^f$ divides $x_1^cx_2^d$ in $\sg$, a contradiction to $\x^c \xx^d \in \invab$.
    Hence $d \leq f$. 

    (3) can be proven using a similar argument to that of (2).

    The proof of (4) follows from the previous parts.
\end{proof}


We will now explicitly construct some elements in $\invab$.

\begin{lemma}\label{inv-divalg}
    Let $0 < b \leq b\inv < p$ and express $p$ as 
    \begin{align*}
        p &= bq +r
        \\
        p &= b\inv s + t.
    \end{align*}
    via the division algorithm. Then $\invab$ contains
    \begin{align*}
        \{ \x^{p-kb}\xx^k \mid 0 \leq k \leq q \} 
        \cup
        \{ \x^{m}\xx^{p-mb\inv} \mid 0 \leq m \leq s \}.
    \end{align*}
\end{lemma}

\begin{proof}
    Suppose $\x^c \xx^d \in \invab$ such that $0 \leq d \leq q$. 
    Thus $c + bd \equiv_p 0$, or equivalently, $c \equiv_p p-bd$.
    We will show $c = p-db$. 
    First note that if $d = 0$ then $c= p$ since $c \equiv_p p-bd \equiv_p p$ and $c \leq p$ by Lemma \ref{invariant-lemmas} (1).
    Now consider when $0 < d \leq q$. 
    This inequality with $0 \leq p-qb$ provides $0 \leq p-db < p$.
    So $c \equiv_p p-db$ where $0 \leq c \leq p$ and $0 \leq p-db < p$.
    Hence $c = p-db$ necessarily follows.

    Now consider an arbitrary monomial $\x^{p-kb}\xx^k$ for some $0 \leq k \leq q$. 
    We know $\x^{p-kb}\xx^k$ lives in $\sg$ since $(p-kb)+ bk \equiv_p 0$. 
    Therefore there exists some generator $\x^c \xx^d \in \invab$ and invariant $g \in \sg$ such that $\x^{p-kb}\xx^k = \x^c \xx^d \cdot g$. 
    This relation requires $c \leq p-kb$ and $d \leq k$. 
    In particular, since $d \leq k$ and $k \leq q$, then $d \leq q$.
    So as $\x^c \xx^d \in \invab$ with $d \leq q$, then $c = p-db$ by the previous paragraph.
    Thus the inequality $c \leq p-kb$ becomes $p-db \leq p-kb$, which yields $k \leq d$.
    Since $d \leq k$ also holds, then $d = k$.
    Thus $\x^{p-kb}\x^k = \x^{p-db}\xx^d = \x^c\xx^d$ where $\x^c \xx^d \in \invab$.
    Hence $\{\x^{p-kb}\xx^k \mid 0 \leq k \leq q\} \subseteq \invab$.

    A similar argument proves $\{\x^{m}\xx^{p-mb\inv} \mid 0 \leq m \leq s \} \subseteq \invab$.
\end{proof}

\begin{rem}
    By Lemma \ref{inv-divalg}, for any $p$ and $b$, the monomials $\x^p , \x^{p-b} \xx, \x \xx^{p-b\inv}, \xx^p$ are always contained in $\invab$. These monomials are distinct when $b \neq p-1$.
    We are interested in the case when $\invab$ is precisely equal to
    $\{ \x^p , \x^{p-b} \xx, \x \xx^{p-b\inv}, \xx^p \}$.
\end{rem}

\begin{example}
    Let $G = \ZZ/17\ZZ$ and $b=10$. Thus $b\inv = 12$ and 
    $$\inva_{1,10} = \{\x^{17}, \x^7 \xx, \x^4 \xx^3, \x \xx^5, \xx^{17}\}.$$
    In reference to Lemma \ref{inv-divalg}, $p=bq+r=b\inv s+t$ for $q=1,r=7$ and $s=1,t=5$. Thus
    \begin{align*}
        &\{ \x^{p-kb}\xx^k \mid 0 \leq k \leq q \} 
        \cup
        \{ \x^{m}\xx^{p-mb\inv} \mid 0 \leq m \leq s \}
        \\
        &=
        \{ \x^{17-10k}\xx^k \mid 0 \leq k \leq 1 \} 
        \cup
        \{ \x^{m}\xx^{17-12m} \mid 0 \leq m \leq 1 \}
        \\
        &=
        \{\x^{17}, \x^7 \xx\}
        \cup
        \{\xx^{17}, \x \xx^5 \}
        \\
        &=
        \{ \x^{17}, \x^7 \xx, \x \xx^5, \xx^{17} \}.
    \end{align*}
    This is an example where $\invab$ is properly contained the union of sets given in the lemma. 
    In Example \ref{ex1}, the two are equal.
\end{example}

\begin{lemma}\label{inv-intersection}
    Let $1 < b < p$ and express $p$ as
    \begin{align*}
        p &= bq +r
        \\
        p &= b\inv s + t.
    \end{align*}
    via the division algorithm. Then
    $$\{ x_1^{p-kb}x_2^k \mid 0 \leq k \leq q \} \cap \{x_1^m x_2^{p-mb\inv} \mid 0 \leq m \leq s-1\} = \emptyset.$$
\end{lemma}

\begin{proof}
    It suffices to prove $r \geq s$ since for $0 \leq k \leq q$, $p-kb$ ranges from $r$ to $p$.
    By way of contradiction suppose $r<s$. 
    Thus $\x^r \xx^q, \x^r \xx^{p-rb\inv}, \x^s \xx^t \in \invab$ by Lemma \ref{inv-divalg}.
    Since both $\x^r \xx^q$ and $\x^r \xx^{p-rb\inv}$ are contained in $\invab$ then $q = p-rb\inv$ by Lemma \ref{invariant-lemmas} (4).
    Therefore $p = b\inv r + q$.
    Now note that since $b > 1$ then $bb\inv = pk+1$ for some $k \geq 1$. 
    This yields $bb\inv \geq p+1 > p = bq+r > bq$, and so $b\inv > q$.
    So we have shown $p = b\inv r +q$ where $q < b\inv$ and $r < s$.
    This contradicts the expression $p = b\inv s + t$ via the division algorithm.
    Therefore $r \geq s$ indeed holds.
\end{proof}


\begin{rem}
    Note that Lemma \ref{inv-intersection} requires $b > 1$ since in the case $b=1$,
    \begin{align*}
        \{x_1^{p-kb}x_2^k \mid 0 \leq k \leq q \}  = \{x_1^{p-k}x_2^k \mid 0 \leq k \leq p \}
    \end{align*}
    and
    \begin{align*}
        \{x_1^m x_2^{p-mb\inv} \mid 0 \leq m \leq s-1\} = \{ x_1^m x_2^{p-m} \mid 0 \leq m \leq p-1\},
    \end{align*}
    which have nontrivial intersection. Let us prove precisely what $\inva_{1,1}$ is.
\end{rem}

\begin{prop}\label{b=1}
    $\inva_{1,1} = \{x_1^{p-k}x_2^k \mid 0 \leq k \leq p\}$.
\end{prop}

\begin{proof}
    Let $x_1^cx_2^d \in \inva_{1,1}$. 
    So $c + d \equiv_p 0$, or equivalently $c \equiv_p p-d$.
    By Lemma \ref{invariant-lemmas} (1), $c,d \leq p$.
    That is,  $c \equiv_p p-d$ where $0 \leq c,p-d \leq p$, and so $c= p-d$.
    Thus $\x^c \xx^d = \x^{p-d}\xx^d$. 
    So $\inva_{1,1} \subseteq \{x_1^{p-k}x_2^k \mid 0 \leq k \leq p\}$. 
    Lemma \ref{inv-divalg} provides the other containment $\{x_1^{p-k}x_2^k \mid 0 \leq k \leq p\} \subseteq \inva_{1,1}$.
\end{proof}

\begin{rem}
    Note that $S^G_{1,1} \cong \CC[x_1^p, x_1^{p-1}x_2, \dots, x_1x_2^{p-1}, x_2^p]$ is precisely the well-understood $p^{th}$ Veronese ring in 2 variables.
\end{rem}

\begin{prop}\label{p-1 inv}
    $\inva_{1,p-1} = \{x_1^p, x_1x_2, x_2^p\}$.
\end{prop}

\begin{proof}
    Lemma \ref{inv-divalg} gives us $\{x_1^p, x_1x_2, x_2^p\} \subseteq \inva_{1,p-1}$. 
    To prove the other containment suppose $x_1^cx_2^d \in \inva_{1,p-1}$ with $0< c,d < p$. 
    By comparing $x_1^cx_2^d$ with $x_1x_2$ in $\inva_{1,p-1}$, Lemma \ref{invariant-lemmas} (2) tells us $d \leq 1$ since $c \geq 1$.
    Hence $d = 1$ and so $x_1^cx_2^d = x_1x_2$. 
    This proves the other containment.
\end{proof}

\begin{lemma}\label{more than 4 inv}
    If $1 < b < \frac{p-1}{2}$ then $|\invab| > 4$.
\end{lemma}

\begin{proof}
    Since $1 < b < \frac{p-1}{2}$, then $p=bq+r$ with $q \geq 2$. 
    By Lemmas \ref{inv-divalg} and \ref{inv-intersection}, $|\invab| \geq q+1+s$. 
    If $q > 2$, then $|\invab| > 2+1+1 = 4$. 
    If $q=2$ then $\invab$ contains $\{x_1^p, x_1^{p-b}x_2, x_1^{p-2b}x_2\} \cup \{x_2^p, x_1x_2^{p-b\inv} \}$ by Lemma \ref{inv-divalg}. 
    Since $b < \frac{p-1}{2}$ then $x_1^{p-2b}x_2$ and $x_1x_2^{p-b\inv}$ are distinct and hence $|\invab| \geq 5$.
    Therefore in either case $|\invab| > 4$.
\end{proof}

Hence $|\invab| = 4$ can only hold when $b \geq \frac{p-1}{2}$.

\begin{prop}\label{pk+1 more than 4 invs}
    If $(p-b)(p-b\inv) = pk+1$ with $k>1$ then $|\invab|>4$.
\end{prop}

\begin{proof}
    Suppose $(p-b)(p-b\inv) = pk+1$ with $k>1$. 
    We first note that $b \neq \frac{p-1}{2}, p-1$. This is because if $b = \frac{p-1}{2}$ then $b\inv = p-2$ and so
    $$(p-b)(p-b\inv) = \left(p - \frac{p-1}{2}\right)(p- (p-2)) = \left( \frac{p+1}{2}\right)\cdot 2 = p+1.$$ 
    If $b=p-1$ then $b\inv = p-1$ and hence $(p-b)(p-b\inv) = 1\cdot 1 = 1$.

    Now in the case $1 < b < \frac{p-1}{2}$, we immediately have $|\invab| > 4$ by Lemma \ref{more than 4 inv}. 
    If $b=1$ then $(p-b)(p-b\inv) = (p-1)^2 = p (p-2)+1$, which equals $pk+1$ with $k>1$ when $p>3$. 
    In which case, by Proposition \ref{b=1}, $|\inva_{1,1}| = p+1 > 3+1 = 4$.
    
    Therefore we have the $\frac{p-1}{2} < b < p-1$ case remaining.
    By Lemma \ref{inv-divalg}, we know $|\invab| \geq 4$ as $\invab$ contains $\{ x_1^p, x_1^{p-b}x_2, x_1x_2^{p-b\inv}, x_2^p\}$. 
    To prove $|\invab| > 4$, we will construct an element of $\sg$ that cannot be generated by these four invariants.
    Consider the monomial $x_1^{b+1}x_2^{p-b\inv -1}$. 
    Since 
    $$(b+1) + b(p-b\inv-1) \equiv_p b + 1 + bp -1 - b \equiv_p 0$$
    then $x_1^{b+1}x_2^{p-b\inv -1} \in \sg$. 
    Now by way of contradiction, suppose $x_1^{b+1}x_2^{p-b\inv -1} \in \langle x_1^p, x_1^{p-b}x_2, x_1x_2^{p-b\inv}, x_2^p \rangle$ in $\sg$. 
    Since $b+1,p-b\inv - 1 < p$ and $p-b\inv -1 < p-b\inv$ then $x_1^{b+1}x_2^{p-b\inv -1} \notin \langle x_1^p, x_2^p, x_1x_2^{p-b\inv} \rangle$. 
    Therefore $x_1^{b+1}x_2^{p-b\inv -1} \in \langle x_1^{p-b}x_2\rangle$ and so 
    \begin{align*}
        x_1^{b+1}x_2^{p-b\inv -1} &= (x_1^{p-b}x_2)^n \quad \text{for some } n \geq 1
        \\
        &=
        x_1^{n(p-b)} x_2^n.
    \end{align*}
    Hence $n(p-b) = b+1$ and $n = p-b\inv -1$. 
    As a result,
    \begin{align*}
        b+1 &= n(p-b)
        \\
        &= (p-b\inv -1)(p-b)
        \\
        &=
        (p-b\inv)(p-b) - (p-b)
        \\
        &=
        pk+1 - (p-b)
        \\
        &=
        p(k-1) + (b+1).
    \end{align*}
    That is, $p(k-1) + (b+1) = b+1$. 
    Hence $p(k-1) = 0$, a contradiction as $k > 1$.
    Therefore
    $$x_1^{b+1}x_2^{p-b\inv -1} \in \sg \setminus \langle x_1^p, x_2^p, x_1x_2^{p-b\inv} \rangle,$$
    and so there must exist at least one other monomial $x_1^cx_2^d \in \invab$ such that 
    \\
    $x_1^{b+1}x_2^{p-b\inv -1} = x_1^cx_2^d \cdot g$ for some $g \in \sg$. Thus $|\invab| > 4$.
\end{proof}

We can now provide the proof of Theorem \ref{thm2}.

\begin{proof}[Proof of Theorem \ref{thm2}]
    First we prove (2) implies (1). 
    Suppose $(p-b)(p-b\inv) =p+1$.
    Then $b \neq p-1$, because if this were the case, then $(p-b)(p-b\inv)$ would equal  1.
    Therefore $x_1^p, x_2^{p-b}x_2, x_1x_2^{p-b\inv}, x_2^p$ are distinct monomials, and by Lemma \ref{inv-divalg}, we know these are elements of $\invab$.
    By way of contradiction suppose there exists another monomial $x_1^cx_2^d \in \invab$. 
    Hence by Lemma \ref{invariant-lemmas} (2),(3), $1<c<p-b$ and $1< d < p-b\inv$.
    By definition of $\invab$, we know $c + bd \equiv_p 0$, and thus $c \equiv_p (p-b)d$. 
    Observe that $(p-b)d < (p-b)(p-b\inv) = p+1$.
    That is, $(p-b)d < p+1$, or equivalently $(p-b)d \leq p$.
    As $c \equiv_p (p-b)d$ where both $0<c< p$ and $0 < (p-b)d \leq p$, then $c = (p-b)d$. Hence $c > p-b$, a contradiction to $1 < c < p-b$.
    Therefore $|\invab| =4$ necessarily holds.

    Now we prove (1) implies (2). 
    Suppose $\sg$ is generated by 4 invariants. Proposition \ref{pk+1 more than 4 invs} implies that $(p-b)(p-b\inv)$ equals $1$ or $p+1$. 
    However, if $(p-b)(p-b\inv) = 1$ then $p-b = p-b\inv = 1$. 
    In this case, $b = p-1$ and Proposition \ref{p-1 inv} shows that $S^G_{1,p-1}$ is generated by 3 invariants. 
    Thus, this case does not occur, which implies $(p-b)(p-b\inv) = p+1$.
    This completes the proof.
\end{proof}

\begin{rem}
    Note for any odd prime $p$, that $(p-b)(p-b\inv) = p+1$ is always satisfied by $b = \frac{p-1}{2}$.
\end{rem}


\medskip

\section{Maximal Minor Presentation for the Codimension 2 Case}\label{section4}

We now determine the free resolution in the case when $\sg$ is generated by four invariants.
In particular, the Hilbert-Burch Theorem guarantees that in this instance, the free resolution will be presented by maximal minors, and we compute a corresponding matrix for this presentation.

\begin{definition}
    For a monomial $x_1^cx_2^d$, define $log_{x_1}(x_1^c x_2^d) = c$ and $log_{x_2}(x_1^cx_2^d) = d$.
\end{definition}

\begin{prop}\label{4 INV KERNEL}
    If $S^G_{1,b}$ is generated by 4 invariants then 
    $$\ker \varphi_{1,b} =  \langle y_1^{p-b\inv}-y_0y_2, y_2^{p-b}-y_1y_3, y_1^{p-b\inv-1}y_2^{p-b-1} - y_0y_3 \rangle .$$
\end{prop}

\begin{proof} 
    For shorthand, let $I = \langle y_1^{p-b\inv}-y_0y_2, y_2^{p-b}-y_1y_3, y_1^{p-b\inv-1}y_2^{p-b-1} - y_0y_3 \rangle$.
    As $S^G_{1,b}$ is generated by 4 invariants then $\invab = \{ x_1^p, x_1^{p-b}x_2, x_1x_2^{p-b\inv}, x_2^p \}$, as shown in the proof of Theorem \ref{thm2}. Equivalently, $(p-b)(p-b\inv)=p+1$.
    By Theorem 65.5 of \cite{peeva}, $\ker \varphi_{1,b}$ is generated by
    $$\left\{ y_0^{v_0}y_1^{v_1}y_2^{v_2}y_3^{v_3} - y_0^{w_0}y_1^{w_1}y_2^{w_2}y_3^{w_3} 
    \Bigg|
    \renewcommand*{\arraystretch}{0.86}
    \begin{pmatrix}
        v_0 \\ v_1 \\ v_2 \\ v_3
    \end{pmatrix}, \begin{pmatrix}
        w_0 \\ w_1 \\ w_2 \\ w_3
    \end{pmatrix} \in \mathbb{N}^4, \;
    A \renewcommand*{\arraystretch}{0.86}
    \begin{pmatrix}
        v_0 \\ v_1 \\ v_2 \\ v_3
    \end{pmatrix} = A \begin{pmatrix}
        w_0 \\ w_1 \\ w_2 \\ w_3
    \end{pmatrix}
    \right\}$$
    where $A = \renewcommand*{\arraystretch}{1.2}
    \begin{pmatrix}
        p & p-b & 1 & 0
        \\
        0 & 1 & p-b\inv & p
    \end{pmatrix}$. Hence $y_1^{p-b\inv} - y_0y_2 \in \ker \varphib$ since
    \begin{align*}
        \renewcommand*{\arraystretch}{0.86}
        \begin{pmatrix}
            p & p-b & 1 & 0
            \\
            0 & 1 & p-b\inv & p
        \end{pmatrix}
        \begin{pmatrix}
            0 \\ p-b\inv \\ 0 \\ 0
        \end{pmatrix}
        =
        \begin{pmatrix}
            p+1 \\ p-b\inv
        \end{pmatrix}
        =
        \begin{pmatrix}
            p & p-b & 1 & 0
            \\
            0 & 1 & p-b\inv & p
        \end{pmatrix}
        \begin{pmatrix}
            1 \\ 0 \\ 1 \\ 0
        \end{pmatrix}.
    \end{align*}
    Similar computation proves the other two binomials are contained in $\ker \varphib$, and thus $I \subseteq \ker \varphi_{1,b}$.

    To prove the other containment, by way of contradiction suppose there exists some $f \in \ker \varphib \setminus I$. 
    Let $G$ be a Gröbner basis of $I$ with respect to reverse lexicographic order with $y_i > y_{i+1}$ for $0 \leq i \leq 2$. 
    Note that for the reduction $\overline{f}$ of $f$ by $G$, it follows that $\varphib (\overline{f}) = \varphib(f) = 0$ since $G \subseteq I \subseteq \ker \varphib$.
    Also since $\overline{f} \in \ker \varphib$ then 
    $$\overline{f} = y_0^{v_0}y_1^{v_1}y_2^{v_2}y_3^{v_3} - y_0^{w_0}y_1^{w_1}y_2^{w_2}y_3^{w_3}$$
    for some $v_i,w_i \geq 0$ and hence $\log_{x_1}( \varphib(y_0^{v_0}y_1^{v_1}y_2^{v_2}y_3^{v_3})) = \log_{x_1}(\varphib(y_0^{w_0}y_1^{w_1}y_2^{w_2}y_3^{w_3}))$.
    That is, $pv_0 + \log_{x_1}( \varphib(y_1^{v_1}y_2^{v_2})) = pw_0 + \log_{x_1}( \varphib(y_1^{w_1}y_2^{w_2}))$. Additionally by definition of $G$, no monomial of $\overline{f}$ is divisible by any of the terms in 
    $$\{y_1^{p-b\inv}, y_2^{p-b}, y_1^{p-b\inv-1}y_2^{p-b-1} \}.$$
    Hence $y_1^{v_1}y_2^{v_2}, y_1^{w_1}y_2^{w_2} \in \Lambda$ where $\Lambda$ is given by
    $$\Lambda = \{ y_1^q y_2^r \mid 0 \leq q \leq p-b\inv -1, \; 0 \leq r \leq p-b-1, (q,r) \neq (p-b\inv-1, p-b-1)\}.$$
    We aim to prove $(v_0,v_1,v_2,v_3)= (w_0,w_1,w_2,w_3)$ by showing there is a bijection between $\Lambda$ and $\{0,1,\dots,p-1\}$. 
    First note that for any $0 \leq q \leq p-b\inv-1$ and $0 \leq r \leq p-b-1$ with $(q,r) \neq (p-b\inv-1,p-b-1)$,
    $$0 \leq (p-b)q+r < (p-b)(p-b\inv-1) + (p-b-1) = (p-b)(p-b\inv) - 1 = (p+1)-1 = p.$$
    Thus we can define the map
    \begin{align*}
        \Lambda &\rightarrow \{0,1,\dots,p-1\}
        \\
        y_1^qy_2^r &\mapsto \log_{x_1}(\varphib(y_1^q y_2^r)) = (p-b)q+r.
    \end{align*}
    As $0 \leq q \leq p-b\inv-1$ and $0 \leq r \leq p-b-1$, the map is injective. 
    Now consider $z \in \{0, 1, \dots, p-1\}$. 
    By the division algorithm, we can express $z = (p-b)q+r$ for some $0 \leq r \leq p-b-1$.
    Further observe
    $$(p-b)(p-b\inv) = p+1 > p-1 \geq z = (p-b)q+r \geq (p-b)q.$$
    Hence $0 \leq q \leq p-b\inv -1$. 
    We also note that $(q,r) \neq (p-b\inv-1, p-b-1)$, otherwise $z = (p-b)(p-b\inv-1) + (p-b-1) = p$. 
    Thus the map is indeed a bijection.
    Revisiting $pv_0 + \log_{x_1}( \varphib(y_1^{v_1}y_2^{v_2})) = pw_0 + \log_{x_1}( \varphib(y_1^{w_1}y_2^{w_2}))$, we can now conclude $(v_0,v_1,v_2) = (w_0, w_1,w_2)$.
    To prove $v_3= w_3$, note that we likewise have 
    $\log_{x_2}( \varphib(y_0^{v_0}y_1^{v_1}y_2^{v_2}y_3^{v_3})) = \log_{x_2}(\varphib(y_0^{w_0}y_1^{w_1}y_2^{w_2}y_3^{w_3})).$
    That is, 
    $$v_1 + (p-b\inv)v_2 + pv_3 = w_1 + (p-b\inv)w_2 + pw_3.$$ 
    As $(v_0,v_1,v_2) = (w_0, w_1,w_2)$, then $v_3 = w_3$ follows. 
    Therefore $\overline{f} = 0$ and hence $f \in G \subseteq I$. 
    This contradicts $f \in \ker \varphib \setminus I$ and so $\ker \varphib = I$ necessarily holds.
\end{proof}

\begin{cor}\label{4 inv res}
    If $\sg$ is generated by 4 invariants then the free resolution of $R^1/\ker \varphib$ is
    $$R^1
    \xlongleftarrow{[\ker \varphi_{1,b}]}
    \renewcommand*{\arraystretch}{0.9}
    \begin{tabular}{ c }
        $R^1(-(2p-b\inv +1))$
        \\
        $\bigoplus$
        \\
        $R^1(-(2p-b+1))$
        \\
        $\bigoplus$
        \\
        $R^1(-2p)$
    \end{tabular}
    \xlongleftarrow{
    \begin{bmatrix}
        -y_3 & -y_2^{p-b-1}
        \\
        -y_1^{p-b\inv-1} & -y_0
        \\
        y_2 & y_1
    \end{bmatrix}
    } 
    \renewcommand*{\arraystretch}{0.9}
    \begin{tabular}{ c }
        $R^1(-(3p-b\inv +1))$
        \\
        $\bigoplus$
        \\
        $R^1(-(3p-b+1))$
    \end{tabular}
    $$
\end{cor}

\begin{proof}
    By Proposition \ref{4 INV KERNEL}, $\ker \varphib = \langle y_1^{p-b\inv}-y_0y_2, y_2^{p-b}-y_1y_3, y_1^{p-b\inv-1}y_2^{p-b-1} - y_0y_3\rangle$. In particular, $\ker \varphib$ is generated by the maximal minors of  
    $$\begin{bmatrix}
        -y_3 & -y_2^{p-b-1}
        \\
        -y_1^{p-b\inv-1} & -y_0
        \\
        y_2 & y_1
    \end{bmatrix}.$$
    Since codim $\ker \varphib = 2$ and $R/\ker \varphib \cong S^G_{1,b}$ is Cohen-Macaulay, the Hilbert-Burch Theorem (\cite{eisenbudcommalg}, Theorem 20.15) provides the free resolution of $S^G_{1,b}$, which is as proposed.
\end{proof}

\medskip
\section{Case for $(p-b)(p-b\inv) = 2p+1$}\label{section5}

Naturally, we ask if there is an analogous theorem to Theorem \ref{thm2} in the case $(p-b)(p-b\inv) = 2p+1$. In this section, we connect the conditions that yield five generating invariants to the case when $(p-b)(p-b\inv) = 2p+1$. We start with the following proposition.

\begin{prop}\label{2p+1 inv}
    If $(p-b)(p-b\inv) = 2p+1$ then $|\invab| = 5$. 
\end{prop}

\begin{proof}
    Suppose $(p-b)(p-b\inv) = 2p+1$. 
    Therefore $b \neq 1, \; \frac{p-1}{2}, \; p-1$, otherwise $(p-b)(p-b\inv)$ would equal $p(p-2)+1, \; p+1, \; 1$, respectively. 
    By Proposition \ref{pk+1 more than 4 invs}, we know $|\invab | > 4$ with $\{ x_1^p, x_1^{p-b}x_2, x_1x_2^{p-b\inv}, x_2^p \} \subsetneq \invab$. 
    We will prove what the fifth invariant is separately for $1 < b < \frac{p-1}{2}$ and $\frac{p-1}{2} < b < p$.

    If $1 < b < \frac{p-1}{2}$ we claim that $\invab = \{ x_1^p, x_1^{p-b}x_2, x_1^{p-2b}x_2^2, x_1x_2^{p-b\inv}, x_2^p\}$. 
    As $1 < b < \frac{p-1}{2}$, then $p=bq+r$ where $q \geq 2$ and so by Lemma \ref{inv-divalg}, $\invab$ contains $\{ x_1^p, x_1^{p-b}x_2, x_1^{p-2b}x_2^2, x_1x_2^{p-b\inv}, x_2^p\}$. 
    By way of contradiction suppose there exists another monomial $x_1^c x_2^d \in \invab$. 
    Thus by Lemma \ref{invariant-lemmas} (4), $d > 2$ and $c > 1$.
    Then by parts (2) and (3) of the same lemma, $c < p-2b$ and $d < p-b\inv$.
    Furthermore by definition of $\invab$, $c + bd \equiv_p 0$.
    So in particular, $c \equiv_p (p-b)d - p$.
    We will show $c$ is exactly equal to $(p-b)d-p$.
    Note that as $b < \frac{p-1}{2}$ and $d >  2$ then $(p-b)d - p > \left(p-\frac{p-1}{2}\right)\cdot 2 - p = \left( \frac{p+1}{2}\right) \cdot 2 - p = 1$.
    Also since $d < p-b\inv$ and $(p-b)(p-b\inv) = 2p+1$ then $(p-b)d - p < (p-b)(p-b\inv) - p = 2p+1-p = p+1$.
    That is, $(p-b)d - p < p+1$, or equivalently, $(p-b)d - p \leq p$.
    So as $c \equiv_p (p-b)d-p$ where $1 < c < p-2b$ and $1 < (p-b)d-p \leq p$, then $c = (p-b)d-p$ necessarily holds. 
    This equality with $c < p-2b$ and $d>2$ yields
    $p-2b > c = (p-b)d - p > (p-b) \cdot 2 - p = p-2b$.
    That is, $p-2b > p-2b$, a contradiction. Hence $\invab = \{ x_1^p, x_1^{p-b}x_2, x_1^{p-2b}x_2^2, x_1x_2^{p-b\inv}, x_2^p\}$ in the case $1 < b < \frac{p-1}{2}$.

    Now let $\frac{p-1}{2} < b < p-1$. We claim $\invab = \{ x_1^p, x_1^{p-b}x_2, x_1^{\frac{p-b+1}{2}}x_2^{\frac{p-b\inv+1}{2}} , x_1x_2^{p-b\inv}, x_2^p\}$. Note that $\frac{p-b+1}{2}$ and $\frac{p-b\inv+1}{2}$ are in fact integers since $(p-b)(p-b\inv) = 2p+1$ and so $p-b$ and $p-b\inv$ are both odd. Now observe
    \begin{align*}
        \frac{p-b+1}{2} + b \left( \frac{p-b\inv+1}{2} \right) 
        &\equiv_p
        \frac{p+b(p-b\inv)+1}{2}
        \\
        &\equiv_p
        \frac{p+(p-p+b)(p-b\inv)+1}{2}
        \\
        &\equiv_p
        \frac{p+p(p-b\inv)-(p-b)(p-b\inv)+1}{2}
        \\
        &\equiv_p
        \frac{p+p(p-b\inv)-(2p+1)+1}{2}
        \\
        &\equiv_p
        p \left( \frac{p-b\inv-1}{2}\right)
        \\
        &\equiv_p
        0.
    \end{align*}
    Therefore $x_1^{\frac{p-b+1}{2}}x_2^{\frac{p-b\inv+1}{2}} \in S^G_{1,b}$. 
    Note that $\frac{p-b+1}{2} < p-b < p$ and thus $x_1^{\frac{p-b+1}{2}}x_2^{\frac{p-b\inv+1}{2}} \notin \langle x_1^p, x_1^{p-b}x_2 \rangle$. 
    Similarly since $\frac{p-b\inv+1}{2} < p-b\inv < p$, then $x_1^{\frac{p-b+1}{2}}x_2^{\frac{p-b\inv+1}{2}} \notin \langle x_1x_2^{p-b\inv}, x_2^p \rangle$. 
    Thus $x_1^{\frac{p-b+1}{2}}x_2^{\frac{p-b\inv+1}{2}} \notin \langle x_1^p, x_1^{p-b}x_2, x_1x_2^{p-b\inv}, x_2^p \rangle$ in $\sg$ and so there exists at least one other monomial 
    $\x^c \xx^d \in \invab$ that generates $x_1^{\frac{p-b+1}{2}}x_2^{\frac{p-b\inv+1}{2}}$. 
    This requires $c \leq \frac{p-b+1}{2}$. 
    To prove $x_1^{\frac{p-b+1}{2}}x_2^{\frac{p-b\inv+1}{2}} \in \invab$, it suffices to show $c = \frac{p-b+1}{2}$.
    By way of contradiction, suppose $c < \frac{p-b+1}{2}$, or equivalently, $c \leq \frac{p-b-1}{2}$.
    Now consider the monomial $\x^c \xx^{c(p-b\inv)}$.
    Since $c + b[c(p-b\inv)] \equiv_p c + bcp -  c bb\inv \equiv_p 0$, then $\x^c \xx^{c(p-b\inv)} \in \sg$.
    Furthermore, $d \equiv_p c(p-b\inv)$. 
    Additionally observe
    \begin{align*}
        c(p-b\inv)
        &\leq
        \left( \frac{p-b-1}{2} \right)(p-b\inv)
        \\
        &=
        \frac{(p-b)(p-b\inv) - (p-b\inv)}{2}
        \\
        &=
        \frac{(2p+1)-(p-b\inv)}{2}
        \\
        &=
        \frac{p+b\inv+1}{2}
        \\
        &<
        \frac{p + (p-1)+2}{2} \qquad \text{since $b < p-1$ and so $b\inv < p-1$}
        \\
        &= \frac{2p+1}{2}
        \\
        &< p+1.
    \end{align*}
    That is, $c(p-b\inv) < p+1$. As $p$ is prime, we further have $c(p-b\inv) < p$.
    We also note that $d \leq p$ by Lemma \ref{invariant-lemmas} (1).
    So as $d \equiv_p c(p-b\inv)$
    with $d \leq p$ and $c(p-b\inv) < p$, then $d = c(p-b\inv)$.
    Thus $\x^c \xx^d = \x^c \xx^{c(p-b\inv)} = (\x \xx^{p-b\inv})^c \in \langle \x \xx^{p-b\inv} \rangle$.
    That is, $\x \xx^{p-b\inv}$ generates $\x^c \xx^d$ in $\sg$, a contradiction to $\x^c \xx^d \in \invab$.
    Therefore $c = \frac{p-b+1}{2}$ must hold, and so we have the containment $\{ x_1^p, x_1^{p-b}x_2, x_1^{\frac{p-b+1}{2}}x_2^{\frac{p-b\inv+1}{2}} , x_1x_2^{p-b\inv}, x_2^p\} \subseteq \invab$.

    To prove the other containment, by way of contradiction suppose there exists another generating invariant $\x^e \xx^f \in \invab$.
    We note that $e > \frac{p-b+1}{2}$ since if $e < \frac{p-b+1}{2}$ then $\x^e \xx^f$ is generated by $\x \xx^{p-b\inv}$ by the argument used in proving the other containment.
    Thus by Lemma \ref{invariant-lemmas} (2), $f < \frac{p-b\inv+1}{2}$.
    By a similar argument to that of $\x^c \xx^{c(p-b\inv)}$, we conclude $\x^e \xx^f = \x^{f(p-b)} \xx^f \in \langle \x^{p-b} \xx \rangle$.
    Therefore we have reached a contradiction. 
    As a result, $\invab = \{ x_1^p, x_1^{p-b}x_2, x_1^{\frac{p-b+1}{2}}x_2^{\frac{p-b\inv+1}{2}} , x_1x_2^{p-b\inv}, x_2^p\}$ for $\frac{p-1}{2} < b < p-1$. 
    This completes the proof of the proposition.
\end{proof}

\begin{rem}
    We recall here that the converse of Proposition \ref{2p+1 inv} is not true (see Example \ref{ex2}, redefining $a =1, b = 5$).
\end{rem}

\begin{lemma}\label{p-1/3}
    Let $b < \frac{p-1}{2}$. Then $(p-b)(p-b\inv) = 2p+1$ if and only if $\frac{p-1}{3} \in \ZZ$ and $b = \frac{p-1}{3}$.
\end{lemma}

\begin{proof}
    Assume $(p-b)(p-b\inv)=2p+1$. 
    As a result,
    $2p+1 = (p-b)(p-b\inv) > \left(p- \frac{p-1}{2} \right)(p-b\inv) = \left( \frac{p+1}{2}\right)(p-b\inv)$, and so $4p+2 > (p+1)(p-b\inv)$. 
    Consequently $4(p+1) > (p+1)(p-b\inv)$, and so $4 > p-b\inv$. 
    We note that $p-b\inv \neq 1$ otherwise $(p-b)(p-b\inv)$ would equal $1$. 
    Since $(p-b)(p-b\inv)$ is odd, we add that $p-b\inv \neq 2$.
    Thus our only possibility is $p-b\inv = 3$. 
    Therefore $2p+1 = (p-b)\cdot 3$, or equivalently, $3b = p-1$. 
    Thus $\frac{p-1}{3} \in \ZZ$ and $b = \frac{p-1}{3}$.

    Now suppose $\frac{p-1}{3} \in \ZZ$ and $b = \frac{p-1}{3}$. Thus $b\inv = p-3$ since
    $$\left( \frac{p-1}{3} \right)(p-3) \equiv_p  \left( \frac{p-1}{3} \right)(-3) \equiv_p 1-p \equiv_p 1.$$
    Hence $(p-b)(p-b\inv) = \left( p - \frac{p-1}{3}\right)(3) = 3p - (p-1) = 2p+1$.
\end{proof}

\begin{prop}\label{2p+1 lower ker}
    If $b < \frac{p-1}{2}$ and $(p-b)(p-b\inv) = 2p+1$ then
    $$\ker \varphib = \langle y_2^2 - y_1y_3, y_1y_2-y_0y_3, y_3^{b+1}-y_2y_4, y_1^2-y_0y_2, y_2y_3^b-y_1y_4, y_1y_3^b - y_0y_4 \rangle.$$
\end{prop}

\begin{proof}
    For shorthand, let 
    $$I = \langle y_2^2 - y_1y_3, y_1y_2-y_0y_3, y_3^{b+1}-y_2y_4, y_1^2-y_0y_2, y_2y_3^b-y_1y_4, y_1y_3^b - y_0y_4 \rangle.$$
    By Lemma \ref{p-1/3}, $b = \frac{p-1}{3}$ and $b\inv = p-3$.
    Proposition \ref{2p+1 inv} gives us
    $$\invab = \{ x_1^p, x_1^{p-b}x_2, x_1^{p-2b}x_2^2, x_1x_2^{p-b\inv}, x_2^p\} = 
    \{ x_1^p, x_1^{\frac{2p+1}{3}}x_2, x_1^{\frac{p+2}{3}}x_2^2, x_1x_2^{3}, x_2^p \}$$
    Therefore $\ker \varphib$ is generated by
    $$\left\{ y_0^{v_0}y_1^{v_1}y_2^{v_2}y_3^{v_3}y_4^{v_4} - y_0^{w_0}y_1^{w_1}y_2^{w_2}y_3^{w_3}y_4^{w_4}
    \Bigg|
    \renewcommand*{\arraystretch}{0.9}
    \begin{pmatrix}
        v_0 \\ v_1 \\ v_2 \\ v_3 \\ v_4
    \end{pmatrix}, \begin{pmatrix}
        w_0 \\ w_1 \\ w_2 \\ w_3 \\ w_4
    \end{pmatrix} \in \mathbb{N}^5, \;
    A \renewcommand*{\arraystretch}{0.9}
    \begin{pmatrix}
        v_0 \\ v_1 \\ v_2 \\ v_3 \\ v_4
    \end{pmatrix} = A \begin{pmatrix}
        w_0 \\ w_1 \\ w_2 \\ w_3 \\ w_4
    \end{pmatrix}
    \right\}$$
    where $A = \renewcommand*{\arraystretch}{1.2}
    \begin{pmatrix}
        p & \frac{2p+1}{3} & \frac{p+2}{3} & 1 & 0
        \\
        0 & 1 & 2 & 3 & p
    \end{pmatrix}$. 
    Hence $y_2^2 - y_1y_3 \in \ker \varphib$ since
    \begin{align*}
        \renewcommand*{\arraystretch}{1.2}
        \begin{pmatrix}
            p & \frac{2p+1}{3} & \frac{p+2}{3} & 1 & 0
            \\
            0 & 1 & 2 & 3 & p
        \end{pmatrix}
        \renewcommand*{\arraystretch}{0.9}
        \begin{pmatrix}
            0 \\ 0 \\ 2 \\ 0 \\ 0
        \end{pmatrix}
        =
        \renewcommand*{\arraystretch}{1.2}
        \begin{pmatrix}
            \frac{2p+4}{3} \\ 4
        \end{pmatrix}
        =
        \begin{pmatrix}
            p & \frac{2p+1}{3} & \frac{p+2}{3} & 1 & 0
            \\
            0 & 1 & 2 & 3 & p
        \end{pmatrix}
        \renewcommand*{\arraystretch}{0.9}
        \begin{pmatrix}
            0 \\ 1 \\ 0 \\ 1 \\ 0
        \end{pmatrix}.
    \end{align*}
    Similar computation proves $I \subseteq \ker \varphib$.

    As in the proof of Proposition \ref{4 INV KERNEL}, by way of contradiction suppose there exists some $f \in \ker \varphib \setminus I$. 
    Let $G$ be the Gröbner basis of $I$ with respect to reverse lexicographic order with $y_i > y_{i+1}$ for $0 \leq i \leq 3$, and let $\overline{f}$ be the reduction of $f$ by $G$. 
    Therefore no monomial of $\overline{f}$ is divisible by any of the terms in
    $$\{y_2^2, y_1y_2, y_3^{b+1}, y_1^2, y_2y_3^b, y_1y_3^b \}.$$
    Hence
    $$\overline{f} = y_0^{v_0}y_1^{v_1}y_2^{v_2}y_3^{v_3}y_4^{v_4} - y_0^{w_0}y_1^{w_1}y_2^{w_2}y_3^{w_3}y_4^{w_4}$$
    such that $y_1^{v_1}y_2^{v_2}y_3^{v_3}, y_1^{w_1}y_2^{w_2}y_3^{w_3} \in \Lambda$ where $\Lambda$ is given by
    $$\Lambda = \{y_3^r, y_2y_3^s, y_1y_3^s \mid 0 \leq r \leq b, 0 \leq s \leq b-1 \}.$$ 
    From here, it suffices to show there is a bijection between $\Lambda$ and $\{0,1,\dots,p-1\}$.
    Consider
    \begin{align*}
        \{ \log_{x_1} (\varphib(y_3^r)) \mid 0 \leq r \leq b\}
        &=
        \{r \mid 0 \leq r \leq b\}
        \\
        &=
        \{0,1,\dots, p-2b-1\} \quad \small{\text{since $b= \frac{p-1}{3} = p-2b-1$}}
        \\
        \{ \log_{x_1} (\varphib(y_2y_3^s)) \mid 0 \leq s \leq b-1 \}
        &=
        \{ (p-2b)+s \mid 0 \leq s \leq b-1 \}
        \\
        &=
        \{p-2b, p-2b+1,\dots, p-b-1\}
        \\
        \{ \log_{x_1} (\varphib(y_1y_3^s)) \mid 0 \leq s \leq b-1\}
        &=
        \{ (p-b) + s \mid 0 \leq s \leq b-1\}
        \\
        &=
        \{ p-b, p-b+1, \dots, p-1\}.
    \end{align*}
    This provides the bijection between $\Lambda$ and $\{0,1,\dots,p-1\}$ by mapping $m \in \Lambda$ to $\log_{x_1}(\varphib (m))$.
    It follows that $\ker \varphib = I$.
\end{proof}

\begin{prop}\label{2p+1 upper ker}
    If $b > \frac{p-1}{2}$ and $(p-b)(p-b\inv)= 2p+1$ then
    $$\ker \varphib = \langle y_2^2 - y_1y_3, y_1^{\alpha -1}y_2 -y_0y_3, y_3^{\beta}-y_2y_4, y_1^{\alpha}-y_0y_2, y_2y_3^{\beta-1}-y_1y_4, y_1^{\alpha-1}y_3^{\beta-1}-y_0y_4 \rangle$$
    where $\alpha := \frac{p-b\inv+1}{2}$ and $\beta := \frac{p-b+1}{2}$.
\end{prop}

\begin{proof}
    For shorthand, let
    $$I = \langle y_2^2 - y_1y_3, y_1^{\alpha -1}y_2 -y_0y_3, y_3^{\beta}-y_2y_4, y_1^{\alpha}-y_0y_2, y_2y_3^{\beta-1}-y_1y_4, y_1^{\alpha-1}y_3^{\beta-1}-y_0y_4 \rangle.$$
    By Proposition \ref{2p+1 inv},
    $$\invab = \{ x_1^p, x_1^{p-b}x_2, x_1^{\beta}x_2^{\alpha} , x_1x_2^{p-b\inv}, x_2^p\}$$ and thus $\ker \varphib$ is generated by 
    $$\left\{ y_0^{v_0}y_1^{v_1}y_2^{v_2}y_3^{v_3}y_4^{v_4} - y_0^{w_0}y_1^{w_1}y_2^{w_2}y_3^{w_3}y_4^{w_4}
    \Bigg|
    \renewcommand*{\arraystretch}{0.9}
    \begin{pmatrix}
        v_0 \\ v_1 \\ v_2 \\ v_3 \\ v_4
    \end{pmatrix}, \begin{pmatrix}
        w_0 \\ w_1 \\ w_2 \\ w_3 \\ w_4
    \end{pmatrix} \in \mathbb{N}^5, \;
    A \renewcommand*{\arraystretch}{0.9}
    \begin{pmatrix}
        v_0 \\ v_1 \\ v_2 \\ v_3 \\ v_4
    \end{pmatrix} = A \begin{pmatrix}
        w_0 \\ w_1 \\ w_2 \\ w_3 \\ w_4
    \end{pmatrix}
    \right\}$$
    where $A = \renewcommand*{\arraystretch}{1.2}
    \begin{pmatrix}
        p & p-b & \beta & 1 & 0
        \\
        0 & 1 & \alpha & p-b\inv & p
    \end{pmatrix}$. 
    It follows that $I \subseteq \ker \varphib$.

    As in the the proof of Propositions \ref{4 INV KERNEL} and \ref{2p+1 lower ker}, let $f \in \ker \varphib \setminus I$, $G$ be the Gröbner basis of $I$ with respect to reverse lexicographic order with $y_i > y_{i+1}$ for $0 \leq i \leq 3$, and $\overline{f}$ be the reduction of $f$ by $G$. Thus
    $$\overline{f} = y_0^{v_0}y_1^{v_1}y_2^{v_2}y_3^{v_3}y_4^{v_4} - y_0^{w_0}y_1^{w_1}y_2^{w_2}y_3^{w_3}y_4^{w_4}$$ 
    such that $y_1^{v_1}y_2^{v_2}y_3^{v_3}, y_1^{w_1}y_2^{w_2}y_3^{w_3} \in \Lambda$ where
    \[
    \Lambda = 
    \left\{ y_1^r y_3^s, y_1^r y_2 y_3^t, y_1^{\alpha-1}y_3^t \mid 
    0 \leq r \leq \alpha-2, 
    0 \leq s \leq \beta-1, 0 \leq t \leq \beta-2
    \right\}.
    \]
    We again prove there is a bijection between $\Lambda$ and $\{0,1,\dots,p-1\}$.
    First fix some $0 \leq r \leq \alpha -2$ and consider
    \begin{align*}
        \{ \log_{x_1}(\varphib(y_1^ry_3^s)) \mid 0 \leq s \leq \beta-1\} 
        &= \{(p-b)r+s \mid 0 \leq s \leq \beta-1\}
        \\
        &= \{ (p-b)r, (p-b)r+1, \cdots, (p-b)r + (\beta-1)\}
    \end{align*}
    \begin{align*}
        \{ \log_{x_1}(\varphib(y_1^r y_2y_3^t)) \mid 0 \leq t \leq &\beta -2 \}
        = \{(p-b)r + \beta + t \mid 0 \leq t \leq \beta - 2\}
        \\
        &= \{(p-b)r + \beta,  (p-b)r + \beta +1, \dots, (p-b)r + 2\beta-2\}
        \\ 
        \text{(since $2\beta-2 = p-b-1$)} \quad
        &= \{(p-b)r + \beta,  (p-b)r + \beta +1, \dots, (p-b)r + (p-b-1)\}.
    \end{align*}
    Therefore 
    \begin{align*}
        \bigcup_{r=0}^{\alpha-2} &\{\log_{x_1}(\varphib(y_1^ry_3^s)), \log_{x_1}(\varphib(y_1^ry_2y_3^t)) \mid 0 \leq s \leq \beta-1, 0 \leq t \leq \beta-2\}
        \\
        &= \bigcup_{r=0}^{\alpha-2} \{(p-b)r, (p-b)r + 1, \dots, (p-b)r + (p-b-1) \}
        \\
        &= \{ 0, 1, \dots, (p-b)(\alpha-2)+(p-b-1)\}
        \\
        &=
        \{0,1, \dots, (p-b)(\alpha-1)-1\}.
    \end{align*}
    Now observe
    \begin{align*}
        \{ \log_{x_1} (\varphib(y_1^{\alpha-1}y_3^t)) &\mid 0 \leq t \leq \beta-2\}
        =
        \{ (p-b)(\alpha-1) + t \mid 0 \leq t \leq \beta-2\}
        \\
        &=
        \{ (p-b)(\alpha-1), (p-b)(\alpha-1)+1, \dots, (p-b)(\alpha-1)+(\beta-2)\}
        \\
        &=
        \{ (p-b)(\alpha-1), (p-b)(\alpha-1)+1, \dots, p-1\}
    \end{align*}
    where $(p-b)(\alpha-1) + (\beta-2) = p-1$ by the following computation:
    \begin{align*}
        (p-b)(\alpha-1)+(\beta-2)
        &=
        (p-b)\left( \frac{p-b\inv-1}{2}\right) + \left(\frac{p-b-3}{2}\right)
        \\
        &=
        \frac{(p-b)(p-b\inv)-(p-b) + (p-b-3)}{2}
        \\
        &=
        \frac{(2p+1)-(p-b)+(p-b-3)}{2}
        \\
        &=
        \frac{2p-2}{2}
        \\
        &=
        p-1.
    \end{align*}
    Hence there is a bjiection between $\Lambda$ and $\{0,1,\dots,p-1\}$ by mapping $m \in \Lambda$ to $\log_{x_1}(\varphib(m))$. We conclude $\ker \varphib = I$.
\end{proof}

\begin{rem}
    As a result of Propositions \ref{2p+1 lower ker} and \ref{2p+1 upper ker}, we know that if $(p-b)(p-b\inv) = 2p+1$ then $\ker \varphib$ is generated by the maximal minors of
    $$\begin{bmatrix}
        y_0 & y_1 & y_2 & y_3^{b}
        \\
        y_1 & y_2 & y_3 & y_4
    \end{bmatrix}$$
    in the case $b < \frac{p-1}{2}$,
    or by the maximal minors of
    $$\begin{bmatrix}
        y_0 & y_1 & y_2^{\alpha-1} & y_3^{\beta-1}
        \\
        y_1 & y_2 & y_3 & y_4
    \end{bmatrix}$$
    where $\alpha = \frac{p-b\inv+1}{2}$ and $\beta = \frac{p-b+1}{2}$ when $b > \frac{p-1}{2}$.
    We also saw in Proposition \ref{2p+1 inv}, if $(p-b)(p-b\inv) = 2p+1$ then $\sg$ is generated by 5 invariants and so $\dim R = 5$. Therefore $\text{depth} \ker \varphib = \text{codim} \ker \varphib  = \dim R - \dim (R /\ker \varphib) = 5 - 2 = 3$. 
    So as $\text{depth} \ker \varphib = 4-2+1$ where $\ker \varphib$ is generated by the maximal minors of either of the mentioned $2 \times 4$ matrix, then by Corollary A2.60 of \cite{eisenbudgeom}, the free resolution of $\sg$ is the Eagon-Northcott complex of the respective matrix.
\end{rem}

\medskip

\section{Case: 2 slopes}\label{section6}

We now shift from cases that observe $S^G_{1,b}$ for a small amount of generating invariants to, what we consider, the 2 slope case.
This studies when the exponent vectors of $\invab$ lie on the union of two lines.

\begin{definition}
    With our usual setup where $\invab = \{\x^{c_0}\xx^{d_0}, \x^{c_1} \xx^{d_1}, \dots, \x^{c_n} \xx^{d_n} \}$ is arranged in lexicographic order, set
    \begin{align*}
        sl_{\mathit{1,b}}
        = 
        \left \{
        \displaystyle \frac{d_{i+1} - d_i}{c_{i+1} - c_i} \; \Big| \; 0 \leq i \leq n-1
        \right\}
    \end{align*}
    to be the collection of slopes of consecutive exponent vectors of $\invab$.
    We say $\invab$ has \textit{2 slopes} if $|\text{sl}_{1,b}| = 2$.
\end{definition}

\begin{theorem}\label{2 SLOPES}
    Let $0 < b \leq b\inv < p$ and express $p$ as
    \begin{align*}
        p &= bq+r
        \\
        p &= b\inv s +t
    \end{align*}
    via the division algorithm.
    Then $\invab$ has 2 slopes if and only if $r=s$ and $q=t$.
\end{theorem}

\begin{proof}
    Suppose $r=s$ and $q=t$. By Lemma \ref{inv-divalg}, $\invab$ contains

    \medskip
    
    $\{ \x^{p-kb}\xx^k \mid 0 \leq k \leq q \} \cup \{ \x^{m}\xx^{p-mb\inv} \mid 0 \leq m \leq s \} =$
    $$\{ x_1^p, x_1^{p-b}x_2, \dots, x_1^{p-(q-1)b}x_2^{q-1}, x_1^{p-qb}x_2^q = x_1^sx_2^{p-sb\inv}, x_1^{s-1}x_2^{p-(s-1)b\inv}, \dots, x_1x_2^{p-b\inv}, x_2^p\}.$$
    We will prove this is precisely the generating set. 
    Suppose $x_1^cx_2^d \in \invab$. 
    If $d \leq q$ then $x_1^c x_2^d = x_1^{p-db}x_2^d$. 
    If $d > q$ then $c < p-qb$ must hold. 
    Since $p-qb=r=s$ then $c < s$. Hence $x_1^c x_2^d = x_1^c x_2^{p-cb\inv}$. This proves the other containment. 
    Now plotting the invariants gives:
    \begin{center}
        \includegraphics[scale=0.22]{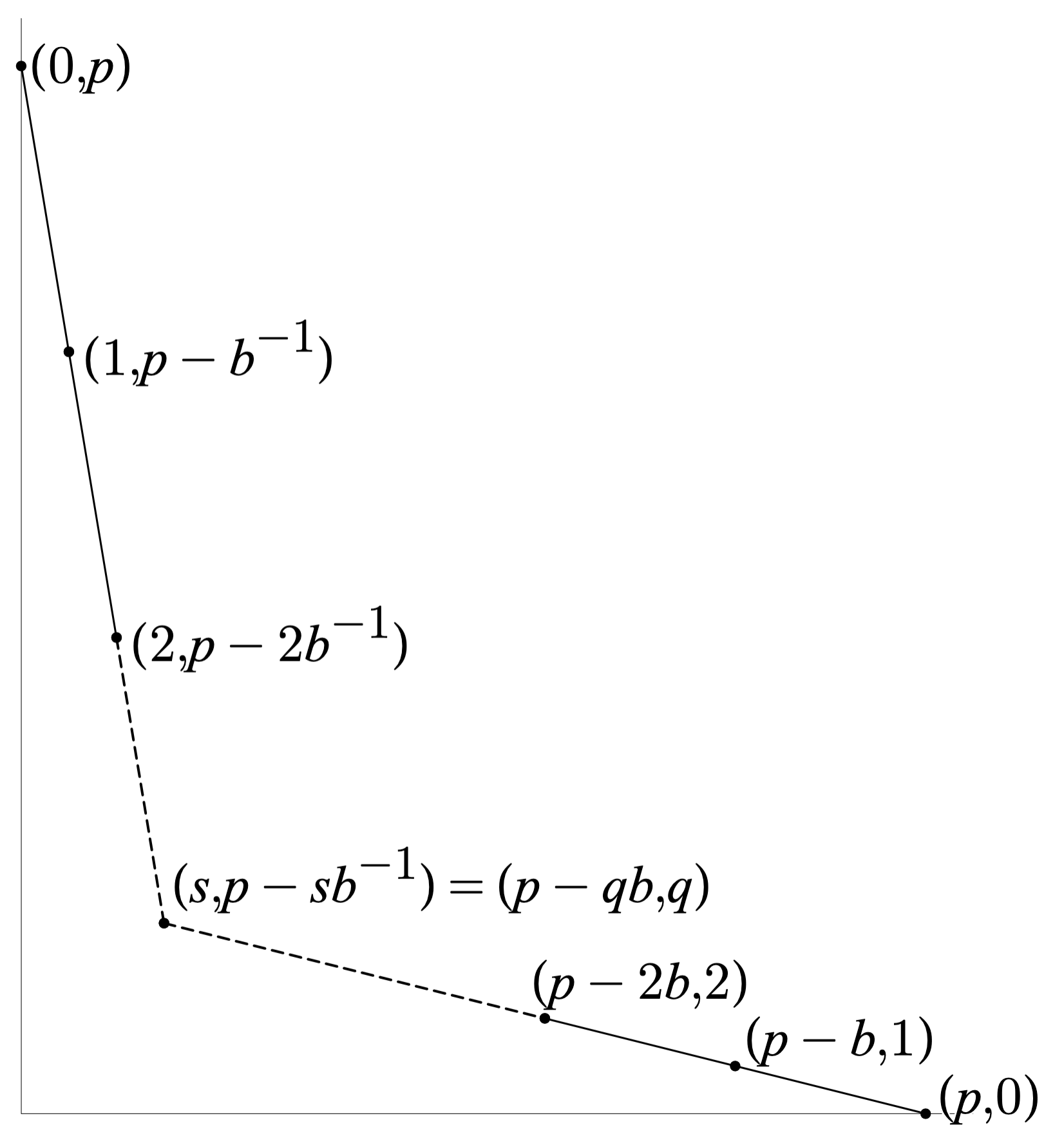}
    \end{center}
    where we can see $\text{sl}_{1,b} = \{-b\inv, -\frac{1}{b}\}$. Thus $\invab$ has 2 slopes.

    Now suppose $\invab$ has 2 slopes.
    We note that $b \neq 1$ since $\text{sl}_{1,1} = \{-1\}$ by Proposition \ref{b=1}.
    So $-b\inv$ and $-\frac{1}{b}$ are distinct, and also belong to $\text{sl}_{1,b}$ since $-b\inv$ corresponds to the line $y-p = -b\inv x$ that $(0,p), (1,p-b\inv), \dots, (s,t)$ lie on, and $-\frac{1}{b}$ corresponds to the line $y= -\frac{1}{b} (x-p)$ that $(r,q), \dots, (p-b,1), (p,0)$ lie on, as depicted in the following graph.
    \begin{center}
        \includegraphics[scale=0.25]{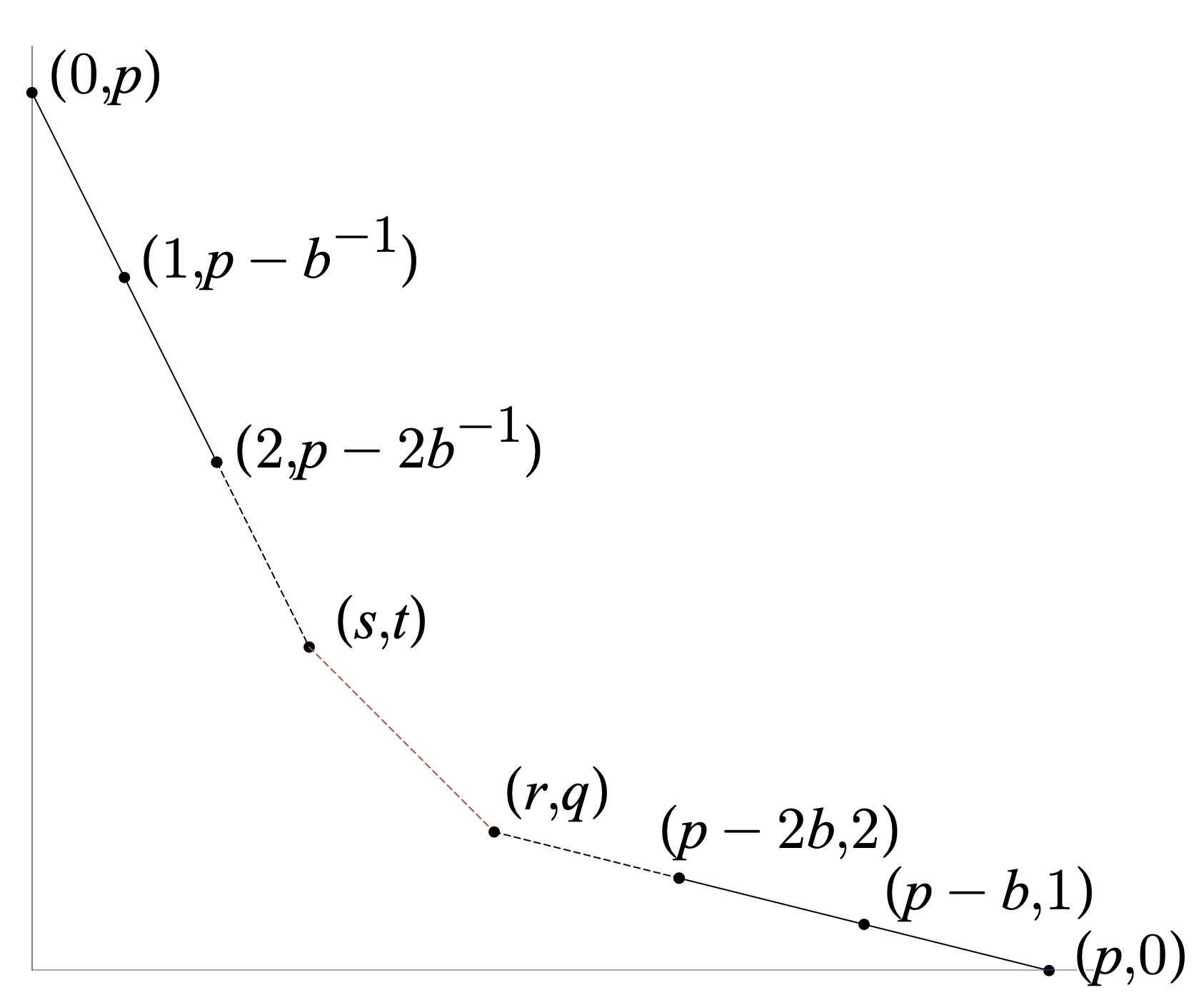}
    \end{center}
    Therefore $\text{sl}_{1,b} = \{-b\inv, -\frac{1}{b}\}$.
    Hence $(s,t)= (r,q)$ must hold otherwise $|\text{sl}_{1,b}| \geq 3$.
\end{proof}

\begin{rem}
    Recall Example \ref{ex3} where $p = 13, b = 4$, and $b\inv = 10$.
    Thus $p =bq+r = b\inv s + t$ for $(r,q) = (1,3) = (s,t)$. The invariants indeed live on precisely two lines.
\end{rem}

\begin{rem}
    For any $p$, Theorem \ref{2 SLOPES} is always satisfied by $b=2, \frac{p-1}{2}, p-1$, with respective inverses $b\inv = \frac{p+1}{2}, p-2, p-1$. Furthermore, if $b$ satisfies Theorem \ref{thm2} then it likewise satisfies Theorem \ref{2 SLOPES} since its generating invariants are $\{\x^p, \x^{p-b} \xx, \x \xx^{p-b\inv}, \xx^p\}$.
\end{rem}

\begin{rem}
    If $\frac{p-1}{2} < b < p-1$ then it cannot satisfy Theorem \ref{2 SLOPES}. This is because $p = bq+r = b\inv s+t$ for $(r,q) = (r,1)$ and $(s,t) = (1,t)$ where $r > 1$. 
\end{rem}

In contrast to the maximal minor presentation of $\ker \varphib$ observed in $\S$\ref{section4} and $\S$\ref{section5}, note that this can fail to hold in the 2 slopes case.
The following example demonstrates this.

\begin{example}
    Let $G = \ZZ/11\ZZ$ and $b = 3$. 
    Thus $b\inv = 4$. 
    So $p = bq+r = b\inv s+t$ for $(r,q) = (2,3) = (s,t)$.
    This satisfies Theorem \ref{2 SLOPES}.
    However,
    \begin{align*}
        \ker \varphi_{1,3}
        = \langle 
        &y_2^2 - y_1y_3, 
        y_3^3 - y_2y_4,
        y_4^2 - y_3y_5,
        y_1y_2 - y_0y_3,
        y_2 y_3^2 - y_1y_4,
        \\
        &\;
        y_3^2 y_4 - y_2y_5,
        y_1^2 - y_0y_2,
        y_1 y_3^2 - y_0y_4,
        y_2y_3y_4 - y_1y_5,
        y_1y_3y_4 - y_0y_5
        \rangle,
    \end{align*}
    one can check that the free resolution of $S^G_{1,3}$ is not given by an Eagon-Northcott complex.
\end{example}

\bigskip

\bibliographystyle{alpha}
\bibliography{references}

\end{document}